\documentclass{amsart}
\usepackage{amsmath}
\usepackage{amsfonts}
\usepackage{amssymb}
\usepackage{graphicx}
\usepackage{float, verbatim}
\usepackage{enumerate}
\usepackage{color}
\usepackage{graphicx}
\usepackage{tikz}
\usetikzlibrary{positioning}
\usetikzlibrary{calc}
\usetikzlibrary{arrows,shapes,backgrounds}

\newcommand{\Haus}{\dim_{\mathrm{H}}}

\newtheorem{thm}{Theorem}[section]

\newtheorem{ques}[thm]{Question}
\newtheorem{lma}[thm]{Lemma}
\newtheorem{cor}[thm]{Corollary}
\newtheorem{defn}[thm]{Definition}

\newtheorem{conj}[thm]{Conjecture}

\newtheorem*{thm*}{Theorem}
\newtheorem*{conj*}{Conjecture}
\newtheorem*{lma*}{Lemma}

\begin{document}
	
	\title{On GILP's group theoretic approach to Falconer's distance problem}
	
	\author{Han Yu}
	\address{Han Yu\\
		School of Mathematics \& Statistics\\University of St Andrews\\ St Andrews\\ KY16 9SS\\ UK \\ }
	\curraddr{}
	\email{hy25@st-andrews.ac.uk}
	\thanks{}

	
	
	\date{}
	
	\dedicatory{}
	
	\begin{abstract}
		In this paper, we follow and extend a group-theoretic method introduced by Greenleaf-Iosevich-Liu-Palsson (GILP) to study finite points configurations spanned by Borel sets in $\mathbb{R}^n,n\geq 2,n\in\mathbb{N}.$ We remove a technical continuity condition in a GILP's theorem in \cite{GILP15}. This allows us to extend the Wolff-Erdogan dimension bound for distance sets to finite points configurations with $k$ points for $k\in\{2,\dots,n+1\}.$ At the end of this paper, we extend this group-theoretic method and illustrate a `Fourier free' approach to Falconer's distance set problem for the Lebesgue measure. We explain how to use tubular incidence estimates in distance set problems. Curiously, tubular incidence estimates are also related to the Kakeya problem. 
	\end{abstract}

	\maketitle
	\section{Introduction}
	Let $n\geq 2$ and $k\in\{2,\dots,n+1\}$ be integers. In this paper we study $k$ points configurations spanned by a subset $F\subset\mathbb{R}^n.$ We start with some definitions. 
	\begin{defn}
		Let $n\geq 2$ and $2\leq k\leq n+1$ be integers. Given a set $F\subset\mathbb{R}^n$, define 
		\[
		\Delta_k(F)=\{(r_{ij},1\leq i<j\leq k)\in\mathbb{R}^{k(k-1)/2}: x_1,\dots,x_k\in F, |x_i-x_j|=r_{ij}, 1\leq i<j\leq k  \}.
		\]
	\end{defn} 
	\begin{defn}\label{Def1}
		Let $F\subset\mathbb{R}^n$ be a Borel set and let $\mu$ be a probability measure supported on $F$. For $g\in\mathbb{O}(n)$, the orthogonal group on $\mathbb{R}^n$, we construct a measure $\nu_g$ as follows,
		\[
		\int_{\mathbb{R}^n} f(z)d\nu_g(z)=\int_F\int_F f(u-gv)d\mu(u)d\mu(v),\forall f\in C_0(\mathbb{R}^n),
		\]
		here $C_0(\mathbb{R}^n)$ is the space of continuous functions with compact support on $\mathbb{R}^n.$ In other words, $\nu_g=\mu*g\mu.$ We also construct a measure $\nu$ on $\Delta_k(F)\subset\mathbb{R}^{k(k-1)/2}$ by
		\[
		\int f(t)d\nu(t)=\int f(|x_1-x_2|,\dots,|x_i-x_j|,\dots,|x_{k-1}-x_k|)d\mu(x_1)\dots d\mu(x_k), \forall f\in C_0(\mathbb{R}^{k(k-1)/2}),
		\]
		where $t$ is a $k(k-1)/2$-vector with entries $|x_i-x_j|$ for $1\leq i<j\leq k.$
	\end{defn}
	In this way, we see that $\nu$ is `the natural measure' supported on $\Delta_k(F).$ In particular, we have $\Haus \nu\leq \Haus \Delta_k(F).$ We will introduce some notions of dimensions in the following section. Notice that our definitions are slightly different than those in \cite{GILP15}. Here we use $k$ to denote the number of vertex of the `simplex structures' we want to count in $F$ while in \cite{GILP15}, $k$ is the order of the simplices. For example, when $k=2$, our definition gives distance sets while the definitions in \cite{GILP15} gives triangle sets.
	
	In this paper, we prove the following result which extends \cite[Theorem 1.3]{GILP15}. The $L^2$ function part was essentially proved in \cite{GILP15} with an additional condition that $\nu_g$ needs to be absolutely continuous with respect to the Lebesgue measure for almost all $g\in\mathbb{O}(n).$
	\begin{thm}\label{HAR1}
		Let $\mu$ be a $s$-Frostman measure with compact support on $\mathbb{R}^n.$ Let $k\in\{2,\dots,n+1\}$ be an integer and $\nu_g,\nu$ be as in Definition \ref{Def1}. We write $\hat{\nu}$ for the Fourier transform of $\nu$. Then for each $\epsilon>0$ there are constants $C,C_\epsilon>0$ such that for all $\delta>0$ we have
		\[
		\int_{B_{\delta^{-1}}(0)} |\hat{\nu}(\omega)|^2 d\omega\leq C\delta^{-n(k-1)}\int\int \nu^{k-1}_g(B_{2.5\delta}(z))d\nu_g(z)dg\leq C_\epsilon\max\{\delta^{-((n-s)(k-1)-\gamma_s+\epsilon)},1\}.
		\]
		If $-((n-s)(k-1)-\gamma_s+\epsilon)>0$ then $\nu$ can be viewed as an $L^2$ function.    Here $\gamma_s$ can be chosen as follows,
		\[
		\gamma_s=
		\begin{cases}
		s & s\in (0,(n-1)/2];\\
		(n-1)/2 & s\in [(n-1)/2,n/2];\\
		(n+2s-2)/4 & s\in [n/2,(n+2)/2];\\
		s-1 & s\in [(n+2)/2,n).
		\end{cases}
		\]
	\end{thm}
	The above result generalizes \cite[Theorem 1.3]{GILP15} in two ways. First, it provides us a good estimate of the growth of $\|\nu_\delta\|_2^2$ with respect to $\delta\to 0,$ which in turn allows us to estimate the Hausdorff dimension of $\Delta_k(F).$ Second, we can drop the technical continuity condition mentioned above. In this way, the above theorem can be seen as an alternative approach to the dimension results of distance sets discussed in \cite[Chapter 15]{Ma2}. We record the Hausdorff dimension estimate as a corollary.
	\begin{cor}\label{HAR2}
		Let $F\subset\mathbb{R}^n,n\geq 2,n\in\mathbb{N}$ be a Borel set with $\Haus F=s.$ Then for each $k\in\{2,\dots,n+1\}$ we have
		\[
		\Haus \Delta_k(F)\geq \min\left\{\frac{k(k-1)}{2}-n(k-1)+s(k-1)+\gamma_s,\frac{k(k-1)}{2}\right\},
		\]
		where $\gamma_s$ is the same quantity as in the statement of Theorem \ref{HAR1}.
	\end{cor}
	We will prove the above result in Section \ref{APP}. For example when $k=3,n=2$ we have \[
	\Haus \Delta_3(F)\geq
	\begin{cases}
	3s-1 & s\in [1/3,1/2];\\
	2s-0.5 & s\in [1/2,1];\\
	2.5s-1 & s\in [1,2].
	\end{cases}
	\] 
	More discussions on this topic will be given in Section \ref{BEFORE}.
	
	We have another consequence from Theorem \ref{HAR1}. We can cover $\mathbb{R}^n$ with closed $\delta$-cubes $\mathcal{K}_\delta$ with disjoint interiors. For each $K\in\mathcal{K}_\delta$ we use $2K$ to denote the $2\delta$-cube with the same centre as $K$. Observe that
	\begin{eqnarray*}
		\int \nu^{k-1}_g(B_\delta(z))\nu_g(B_\delta(z))dz&\leq&\sum_{K\in\mathcal{K}_\delta}\int_K \nu^k_g(B_{\delta}(z))dz\\
		&\leq& \sum_{K\in\mathcal{K}_\delta} \delta^{-n} \sup_{z*\in K} \nu^{k}_g(B_\delta(z*)) \\
		&\leq& \delta^{-n}\sum_{K\in\mathcal{K}_\delta}\int_{2K} \nu^{k-1}_g(B_{2\sqrt{n}\delta}(z)) d\nu_g(z).
	\end{eqnarray*}
	Since $\{2K\}_{K\in\mathcal{K}_\delta}$ covers $\mathbb{R}^n$ with maximal multiplicity $2^{n+1},$ we see that
	\[
	\int \nu^{k-1}_g(B_\delta(z))\nu_g(B_\delta(z))dz\leq \delta^{-n} 2^{n+1} \int \nu^{k-1}_g(B_{2\sqrt{n}\delta}(z))d\nu_g(z).
	\]
	By Theorem \ref{HAR1} and the argument above we see that if $(n-s)(k-1)-\gamma_s<0$,
	\[
	\delta^{-kn}\int\int \nu^{k}_g(B_\delta(z))dz dg\lesssim 1.
	\]
	From here we deduce the following corollary.
	
	\begin{cor}\label{ma}
		Let $\mu$ be a $s$-Frostman measure with compact support on $\mathbb{R}^n.$ Let $k\in\{2,\dots,n+1\}$ be an integer and $\nu_g,\nu$ be as in Definition \ref{Def1}. If $-((n-s)(k-1)-\gamma_s)>0$, then for almost all $g\in\mathbb{O}(n)$, $\nu_g$ is an $L^{k}(\mathbb{R}^n)$ function. In particular, for such $g\in\mathbb{O}(n)$, $\nu_g$ is absolutely continuous with respect to the Lebesgue measure.  Here $\gamma_s$ can be chosen as follows,
		\[
		\gamma_s=
		\begin{cases}
		s & s\in (0,(n-1)/2];\\
		(n-1)/2 & s\in [(n-1)/2,n/2];\\
		(n+2s-2)/4 & s\in [n/2,(n+2)/2];\\
		s-1 & s\in [(n+2)/2,n).
		\end{cases}
		\]
	\end{cor}
	If $n=s$ then $\nu_g$ is an $L^{\infty}$-function for almost all $g\in\mathbb{O}(n).$ For $k=2$, we see that the positivity criterion happens when 
	\[
	s>\frac{n}{2}+\frac{1}{3}.
	\]
	This improves \cite[Lemma 7.1]{Ma2} which requires that $s>(n+1)/2.$ In fact, \cite[Lemma 7.1]{Ma2} is stated in an asymmetric version. However, our method can be adapted to show the corresponding asymmetric results as well. In Section \ref{asy} we give some sketched discussions in this situation. In \cite[Section 7.3]{Ma2}, it was asked whether the following conjecture is true.
	\begin{conj}
		Let $\mu$ be a $s$-Frostman measure with compact support on $\mathbb{R}^n.$ If $s> n/2$ then for almost all $g\in\mathbb{O}(n)$, $\nu_g$ is absolutely continuous with respect to the Lebesgue measure.
	\end{conj}
	
	The proof of Theorem \ref{HAR1} relies heavily on Fourier analysis. Although there are some 'Fourier free' methods which deal with the distance set problem in $\mathbb{R}^2$ (see for example \cite{OP}, \cite{S}), all methods dealing with the Lebesgue measure of distance sets do use Fourier analysis, see \cite[page 62]{Ma2}. In this paper, we also introduce a method which is very close to Guth-Katz's approach (\cite{GK}). We reduce the distance set problem to tubular incidence estimates. Curiously, tubular incidence estimates also have some connections with the Kakeya problem, see Section \ref{dis}.
	
	\section{Notation}
	1. Let $f$ be a function on $\mathbb{R}^n,$ we write $\hat{f}$ for its Fourier transform,
	\[
	\hat{f}(\omega)=\int f(x)e^{-2\pi i (\omega,x)}dx,
	\]
	where $\omega\in\mathbb{R}^n$ and $(\omega,x)$ is the Euclidean inner product between $\omega$ and $x$. Let $\mu$ be a probability measure on $\mathbb{R}^n$ we also write $\hat{\mu}$ for its Fourier transform,
	\[
	\hat{\mu}(\omega)=\int e^{-2\pi i (\omega,x)}d\mu(x).
	\]
	
	2. For each integer $n\geq 1$, we will often need to find a smooth cutoff function $\phi_n$ on $\mathbb{R}^n.$ More precisely, we define $\phi$ to be $1$ on the unit ball and $0$ outside the ball of radius $2$ centred at the origin. Then we can smoothly construct this function $\phi_n.$ When the ambient space is clear, we will write $\phi=\phi_n$ for simplicity. 
	
	Let $\delta>0$ we write $\phi_\delta$ to be the function
	\[
	x\in\mathbb{R}^n \to \delta^{-n} \phi(x\delta^{-1}).
	\]
	Throughout this paper, we never use the symbol $\phi_n$ for the cutoff function we chose on $\mathbb{R}^n$. In this way, no confusions should arise between $\phi_\delta$ and $\phi_n.$ Let $f$ be a function on $\mathbb{R}^n$ we write $f_\delta=f*\phi_\delta.$ Similarly for a measure $\mu$ we write $\mu_\delta=\mu*\phi_\delta.$
	
	3. It is convenient to introduce notions $\approx, \lesssim, \gtrsim$ for approximately equal, approximately smaller and approximately larger. As our estimates always involve scales,  we use $1>\delta>0$ to denote a particular scale. Then for two quantities $f(\delta), g(\delta)$ we define the following:
	\[
	f\lesssim g\iff \exists M>0, \forall \delta>0, f(\delta)\leq M g(\delta).
	\] 
	\[
	f\gtrsim g\iff g\lesssim f.
	\]
	\[
	f\approx g\iff f\lesssim g \text{ and } g\lesssim f.
	\]
	We will use the same symbols for scales tending to $\infty$ as well. More precisely, for $R\in (0,\infty)$, and quantities $f(R), g(R)$, we write
	\[
	f(R)\lesssim g(R)
	\]
	if there is a constant $C>0$ such that $f(R)\leq C g(R)$ for all $R>0.$ Similar meanings can be given to symbols $\gtrsim$ and $\approx.$
	\section{Preliminaries}
	
	\subsection{Hausdorff dimension for sets}
	Let $n\geq 1$ be an integer. Let $F\subset\mathbb{R}^n$ be a Borel set. For any $s\in\mathbb{R}^+$ and $\delta>0$ define the following quantity
	\[
	\mathcal{H}^s_\delta(F)=\inf\left\{\sum_{i=1}^{\infty}(\mathrm{diam} (U_i))^s: \bigcup_i U_i\supset F,\forall i\geq 1, U_i\subset\mathbb{R}^n,\mathrm{diam}(U_i)<\delta\right\}.
	\]
	The $s$-Hausdorff measure of $F$ is
	\[
	\mathcal{H}^s(F)=\lim_{\delta\to 0} \mathcal{H}^s_{\delta}(F).
	\]
	The Hausdorff dimension of $F$ is
	\[
	\Haus F=\inf\{s\geq 0:\mathcal{H}^s(F)=0\}=\sup\{s\geq 0: \mathcal{H}^s(F)=\infty          \}.
	\]
	More details about the Hausdorff dimension can be found in \cite{Fa} and \cite{Ma1}.
	\subsection{Frostman's measure}
	It is known (for example, see \cite[Theorme 2.7]{Ma2}) that if $F$ is a Borel subset of $\mathbb{R}^n$ with $\Haus F=s$, then for any $\epsilon>0$ there is a measure $\mu$ supported in $F$ such that for all $x\in F$ and $r>0$ we have $\mu(B(x,r))\leq r^{s-\epsilon}.$ Such a measure $\mu$ is usually called a $(s-\epsilon)$-Frostman measure.
	
	\subsection{Energy integrals and Hausdorff dimension for measures}
	Let $\mu\in\mathcal{P}(\mathbb{R}^n),$ the space of Borel probability measures on $\mathbb{R}^n.$ For each positive number $t>0$ we define the $t$-energy of $\mu$ to be
	\[
	I_t(\mu)=\int\int \frac{d\mu(x)d\mu(y)}{|x-y|^t}.
	\]
	Via Fourier transform it can be shown that
	\[
	I_t(\mu)=\gamma(n,s) \int |\hat{\mu}(\omega)|^2|\omega|^{t-n}d\omega,
	\]
	where $\gamma(n,s)=\pi^{s-n/2}\Gamma((n-s)/2)/\Gamma(s/2)$ and when $s\in (0,n)$ we have $\gamma(n,s)\in (0,\infty),$ see \cite[Sections 3.4,3.5]{Ma2}. We define the Hausdorff dimension of $\mu$ as follows,
	\[
	\Haus\mu=\sup\{t>0: I_t(\mu)<\infty\}.
	\]
	Let $F\subset\mathbb{R}^n$ be a Borel set, then we have
	\[
	\Haus F=\sup\{t>0: \exists \mu\in\mathcal{P}(F), I_t(\mu)<\infty\}.
	\]
	This implies that if $\mu\in\mathcal{P}(F)$, we have $\Haus \mu\leq \Haus F.$
	
	\subsection{Spherical averages and Wolff-Erdogan's estimate}
	Let $\mu\in\mathcal{P}(\mathbb{R}^n).$ We define the following spherical average for $\hat{\mu},$
	\[
	S(\mu,R)=\int_{S^{n-1}} |\hat{\mu}(R\sigma)|^2 d\sigma,
	\]
	where $d\sigma$ is the normalized Lebesgue measure on $S^{n-1}.$ We have the following deep result on the decay rate of $S(\mu,R)$ as $R\to\infty,$ see \cite{W2}, \cite{E}. The following version is taken from \cite[Theorem 15.7]{Ma2}.
	
	\begin{thm}[Wolff-Erdogan estimate]\label{WE}
		Let $\mu\in\mathcal{P}(\mathbb{R}^n)$ with compact support, for each $s\geq n/2,\epsilon>0,$ there is a positive constant $C(n,s,\epsilon)$ and for all $R>0$ we have
		\[
		S(\mu,R)\leq C(n,s,\epsilon) R^{\epsilon-\gamma_s} I_s(\mu).
		\]
		Here $\gamma_s$ can be chosen as follows,
		\[
		\gamma_s=
		\begin{cases}
		s & s\in (0,(n-1)/2];\\
		(n-1)/2 & s\in [(n-1)/2,n/2];\\
		(n+2s-2)/4 & s\in [n/2,(n+2)/2];\\
		s-1 & s\in [(n+2)/2,n).
		\end{cases}
		\]
	\end{thm}
	Thus, if $\Haus \mu>s$ then we see that $I_s(\mu)<\infty$ and $S(\mu,R)\lesssim R^{-\gamma_s}.$

	\subsection{Distance sets and finite points configurations}\label{BEFORE}
	Let $F\subset\mathbb{R}^n$ be a Borel set. We have defined $\Delta_k(F)$ for all $k\in \{2,\dots,n+1\}.$ A special case is when $k=2.$ In this case we write $D(F)=\Delta_2(F)$ and call it the distance set of $F.$ If $F$ is a finite set in $\mathbb{R}^2$, by a result in \cite{GK} we have
	\[
	\#D(F)\gtrsim \#F/\log \#F.
	\]
	Here for a set $A$ we use $\#A$ to denote the cardinality of $A$. For $F$ being a Borel set with positive Hausdorff dimension, we are interested in whether $D(F)$ has full Hausdorff dimension or ever positive Lebesgue measure. In this direction, we have the following conjecture.
	\begin{conj*}[Falconer's distance conjecture]
		Let $n\geq 2$ be an integer. Let $F\subset\mathbb{R}^n$ be a Borel set with $\Haus F>n/2.$ Then $D(F)$ has positive Lebesgue measure.
	\end{conj*}
	
	See \cite{GIOW18},\cite{KS18},\cite{OP},\cite{S} for some recent results.
	
	A natural generalization of the distance set problem is to consider finite points configurations with more than two points, see \cite{GILP15} and the references therein. When $n=2, k=3$ we meet the problems considering `triangle sets in the plane'. This was studied in \cite{Y18}. In particular, if $\Haus F=s\in (0,2)$ one can show that the lower box dimension of $\Delta_3(F)$ is at least $3s/2.$ Unlike most of the results which follow from harmonic analytic methods, this $3s/2$ bound holds for $s<1$ as well. For distance sets ($k=2,n\geq 2$), one can obtain a similar result which says that the upper box dimension of $D(F)$ is at least $s/n.$ We note here that the $s/n$ bound is often strict, see \cite{FHY}.
	
	\subsection{Orthogonal group, Haar measure}
	For each integer $n\geq 2,$ we denote $\mathbb{O}(n)$ to be the orthogonal group of order $n$ over $\mathbb{R}.$ It can be represented by $n\times n$ real matrices $A$ with $A^T A=I.$ $\mathbb{O}(n)$ is a real compact Lie group of algebraic dimension $n(n-1)/2.$ We associate $\mathbb{O}(n)$ with the normalized Haar measure and we often write
	\[
	\int dg
	\]
	instead of 
	\[
	\int_{\mathbb{O}(n)}dg
	\]
	for simplicity.
	
	\subsection{Group-theoretic energy}
	Let $n\geq 2$ be an integer and $k\in\{2,\dots,n+1\}.$ Let $\mu\in\mathcal{P}(\mathbb{R}^n).$ For each $\delta>0, g\in\mathbb{O}(n)$ we define $k$-\emph{group-theoretic energy} for $\mu$ as scale $\delta$ with respect to $g$ to be
	\[
	E^k(\mu,g,\delta)=\mu^{2k}\{(x_1,\dots,x_k,y_1,\dots,y_k)\in\mathbb{R}^{2kn}: |(x_i-gy_i)-(x_j-gy_j)|\leq \delta\}.
	\]
	Often we can write $E(\mu,g,\delta)$ for $E^k(\mu,g,\delta)$ as the dependence on $k$ will be always assumed. If $A\subset\mathbb{R}^n$ is a finite set and $\mu$ is the normalized counting measure on $A.$ Let $k=2,\delta=0$ we see that
	\[
	E(\mu,g,0)=\mu^{4}\{(x_1,x_2,y_1,y_2)\in A^4: x_1-x_2=g(y_1-y_2) \},
	\]
	which counts the number of quadruples $(x_1,x_2,y_1,y_2)$ of $A$ such that $x_1-x_2=g(y_1-y_2).$ This idea was introduced in \cite{ES10} and it played a crucial role in Guth-Katz's proof of Erd\H{o}s' distance problem, see \cite{GK} and \cite[Section 9]{Guth}. We will discuss more about this in Section \ref{dis}.
	\subsection{AD-regular sets and measures}
	Let $n\geq 1$ be an integer. We say that a Borel measure $\mu\in\mathcal{P}(\mathbb{R}^n)$ is AD-regular with exponent $s\geq 0$ if there is a constant $C$ such that for each $x\in supp(\mu),$ the support of $\mu$, and $r>0$ we have
	\[
	C^{-1}r^s\leq \mu(B_r(x))\leq C r^s.
	\]
	We say that a compact set $F\subset\mathbb{R}^n$ is AD-regular if there is an AD-regular measure $\mu$ such that $\mu(K)>0.$ We will need these notions in Section \ref{dis}.
	\section{An $L^2$ approach to the Hausdorff dimension}\label{L2}
	
	\subsection{Some general results:}
	In this section, we discuss a simple method for estimating the Hausdorff dimension of a Borel probability measure $\mu$ in $\mathcal{P}(\mathbb{R}^n).$ We denote its Fourier transform as $\hat{\mu}.$ It is a continuous function as $\mu$ is compactly supported. In general it is not $L^2$, for otherwise $\mu$ is in fact an $L^2$ function. To measure how far away it is from being $L^2$ we take the following ball average, see also \cite[Section 3.8]{Ma2},
	\[
	A(\mu,R)=\int_{B_R(0)} |\hat{\mu}(\omega)|^2 d\omega.
	\]
	If $\lim_{R\to\infty} A(\mu,R)<\infty$ then $\mu$ can be viewed as an $L^2$ function. In general we expect that $A(\mu,R)$ tends to $\infty$ in a certain speed. If there is a constant $C>0$ and a number $s>0$ such that
	\[
	A(\mu,R)\leq CR^s
	\]
	for all $R>0,$ then we see that for $t\in (0,n)$
	\[
	I_t(\mu)=\int |\hat{\mu}(\omega)|^2|\omega|^{t-n}d\omega=\int_{|\omega|\leq 1} |\hat{\mu}(\omega)|^2|\omega|^{t-n}d\omega+\sum_{j\geq 0} \int_{2^j\leq |\omega|\leq 2^{j+1}}|\hat{\mu}(\omega)|^2|\omega|^{t-n}d\omega.
	\]
	Since $\mu$ is a probability measure, $\hat{\mu}$ is bounded on unit ball. Therefore we see that
	\[
	\int_{|\omega|\leq 1} |\hat{\mu}(\omega)|^2|\omega|^{t-n}d\omega<\infty.
	\]
	For each $j\geq 0$ we have
	\[
	\int_{2^j\leq |\omega|\leq 2^{j+1}}|\hat{\mu}(\omega)|^2|\omega|^{t-n}d\omega\leq A(\mu,2^{j+1})2^{j(t-n)}\leq C 2^{(j+1)s}2^{j(t-n)}=C2^{s} 2^{j(s+t-n)}.
	\]
	If $s+t-n<0$ the sum with respect to $j$ converges and we have
	\[
	I_t(\mu)<\infty.
	\]
	Therefore $\Haus \mu\geq t$ whenever $t<n-s.$ This implies that
	\[
	\Haus \mu\geq n-s.
	\]
	
	In order to study this $L^2$ phenomena more systematically we introduce the following notion of dimension,
	\[
	\dim_{L^2} \mu=n-\limsup_{R\to\infty}\frac{\log A(\mu,R)}{\log R}.
	\]
	There are several other ways of doing this $L^2$ approach. For example we can define
	\[
	A(\mu,R,h)=\int_{B(0,R)}|\hat{\mu}(\omega)|^2h(\omega)d\omega
	\]
	for a weight function $h$ on $\mathbb{R}^n.$ For example if we choose $h(\omega)=|\omega|^{-t}$ for a number $t\geq 0$ we see that
	\[
	A(\mu,R,h)\leq A(\mu,1,h)+\sum_{j\geq 0}\int_{|\omega|\in [2^j,2^{j+1}]} |\hat{\mu}(\omega)|^2 |\omega|^{-t} d\omega\leq A(\mu,1,h)+\sum_{j\geq 0,2^j\leq 2R} 2^{-jt}A(\mu,2^{j+1}).
	\]
	Thus, if $A(\mu,2^{j})\lesssim 2^{uj}$ then we see that
	\[
	\sum_{j\geq 0,2^j\leq 2R}  2^{-jt}A(\mu,2^{j+1})\lesssim\sum_{j\geq 0,2^j\leq 2R}  2^{-jt}2^{ju}\lesssim R^{u-t} 
	\]
	if $u-t>0$ or else the above sum is bounded uniformly for all $R.$ In terms of the $L^2$-dimension we see that if $\dim_{L^2} \mu>n-t$ then
	\[
	\sup_{R} A(\mu,R,h)<\infty,
	\]
	otherwise
	\[
	A(\mu,R,h)\lesssim R^{n-\dim_{L^2}\mu-t}.
	\]
	
	In general $A(\mu,R,|.|^{-t})$ could have a smaller growth exponent. It is interesting to find the infimum among all possible values $s$ such that
	\[
	A(\mu,R,|.|^{-t})\lesssim R^{s}
	\]
	holds for all $R>0.$ More precisely, we consider the following quantity
	\[
	\dim_{L^2,t}\mu=n-t-\limsup \frac{\log A(\mu,R,|.|^{-1})}{\log R}.
	\]
	For $t\geq 0$ we have
	\[
	\dim_{L^2,t}\mu\leq \dim_{L^2}\mu.
	\]
	In general,  it is possible that the above inequality is strict.     We have shown the following result.
	\begin{thm}
		Let $n\geq 1$ be an integer and $\mu\in\mathcal{P}(\mathbb{R}^n)$ be a Borel probaility measure. Then we have
		\[
		\Haus \mu\geq \dim_{L^2}\mu.
		\]
		The function $t\geq 0\to \dim_{L^2,t}\mu$ is non-increasing and bounded from above by $\dim_{L^2}\mu.$
	\end{thm}

	In most cases, it is difficult to estimate $A(\mu,R)$ directly. A useful method is to consider the $L^2$-norm of $\mu_\delta=\mu*\phi_\delta.$ Notice that $\mu_\delta$ is a Schwartz function taking non-negative values. Since $\hat{\mu_\delta}=\hat{\mu}\hat{\phi_\delta}$ and $\hat{\phi_\delta}$ decays very fast outside the ball $B_{\delta^{-1}}(0)$ we see that
	\[
	A(\mu,\delta^{-1})\lesssim\|\mu_\delta\|^2_2=\int \mu^2_\delta(x)dx,
	\]
	where the implicit constant in $\lesssim$ depends only on the choice of the cutoff function $\phi.$
	
	\subsection{Wolff-Erdogan bound for finite points configurations: proof of corollary \ref{HAR2}}\label{APP}
	Before we prove Theorem \ref{HAR1}, let us see how to obtain a Hausdorff dimension estimate. Let $F\subset\mathbb{R}^n$ and $\Haus F=s.$ Then we can choose $(s-\epsilon)$-Frostman measure on $F$ for each $\epsilon>0.$ Then by Theorem \ref{HAR1} together with the discussions above we see that
	\[
	\Haus \Delta_k(F)\geq \frac{k(k-1)}{2}-(n-s)(k-1)+\gamma_s,
	\]
	provided that the RHS is not greater than $k(k-1)/2$, otherwise, $\Delta_k(F)$ has positive Lebesgue measure. For $k=2,$ this result revisits the Wolff-Erdogan-Mattila's bound for the Hausdorff dimension of distance set.
	\section{GILP's lemma and an energy integral estimate}
	First, we introduce a lemma obtained in \cite{GILP15}. 
	\begin{lma}\label{GILP}
		Let $n\geq 2$ and $k\in\{2,\dots,n+1\}$ be integers. Let $\mu\in\mathcal{P}{[0,1]^n}$ and $\nu_g,\nu$ as defined before.  Then there is a constant $C>0$ and we have for all $\delta>0$
		\[
		\int \nu^2_\delta(z)dz\leq C \delta^{-n(k-1)} \int E(\mu,g,\delta)dg,
		\]
		where $\nu_\delta=\nu*\phi_\delta$ is the smoothed version of $\nu$ with scale $\delta>0$ and $E(\mu,g,\delta)$ is the group-theoretic energy of $\mu$ with scale $\delta>0.$
	\end{lma}
	\begin{proof}
		A proof can be found in \cite[Section 2]{GILP15}. 
	\end{proof}

	\begin{lma}\label{ENERGY}
		Let $n\geq 2$ and $k\in\{2,\dots,n+1\}$ be integers. Let $\mu\in\mathcal{P}{[0,1]^n}$ and $\nu_g,\nu$ as defined before. Then for each $\delta>0,g\in\mathbb{O}(n)$ we have
		\[
		E(\mu,g,\delta)\leq \int \nu_g^{k-1}(B_{2.5\delta}(z))d\nu_g(z).
		\]
	\end{lma}
	\begin{proof}
		By putting in definitions we see that the statement of this lemma is equivalent to
		\[
		\mu^{2k}\{(x_1,\dots,x_k,y_1,\dots,y_k)\in\mathbb{R}^{2kn}: |(x_i-gy_i)-(x_j-gy_j)|\leq \delta,1\leq i<j\leq k\}\leq \int \nu^{k-1}_g(B_{2\delta}(z))d\nu_g(z).
		\]
		To prove this, let $x_1,\dots,x_{k-1},y_1,\dots,y_{k-1}$ be fixed, consider the following section
		\[
		\{(x_k,y_k): |(x_i-gy_i)-(x_j-gy_j)|\leq \delta,1\leq i<j\leq k\}.
		\]
		It is easy to see that the above section is contained in
		\[
		E=\{(x_k,y_k): |(x_k-gy_k)-(x_1-gy_1)|\leq \delta\}.
		\]
		We see that the $\mu^{2k}$ measure is now bounded from above by
		\[
		\mu^{2(k-1)}\{(x_1,\dots,x_{k-1},y_1,\dots,y_{k-1})\in\mathbb{R}^{2(k-1)n}: |(x_i-gy_i)-(x_j-gy_j)|\leq \delta,1\leq i<j\leq k-1\}\times
		\]
		\[
		\int 1_E(x_k,y_k)d\mu(x_k)d\mu(y_k).
		\]
		Observe that $1_E(x_k,y_k)=f(x_k-gy_k)$ for $f:z\in\mathbb{R}^n\to f(z)=1_{\{a:|a-(x_1-gy_1)|\leq \delta\}}(z).$ By the definition of $\nu_g$ we see that
		\[
		\int 1_E(x_k,y_k)d\mu(x_k)d\mu(y_k)\leq\nu_g(B_{2\delta}(x_1-gy_1)).
		\]
		If $\nu_g$ does not give positive measure on any spheres then we would get
		\[
		\int 1_E(x_k,y_k)d\mu(x_k)d\mu(y_k)=\nu_g(B_{\delta}(x_1-gy_1)).
		\]
		However, we do not assume this continuity of $\nu_g$ and we only have an upper bound.    We can do the above step $k-1$ times and by Fubini's theorem we see that 
		\[
		\mu^{2k}\{(x_1,\dots,x_k,y_1,\dots,y_k)\in\mathbb{R}^{2kn}: |(x_i-gy_i)-(x_j-gy_j)|\leq \delta,1\leq i<j\leq k\}
		\]
		\[
		\leq \int \nu^{k-1}_g(B_{2\delta}(x_1-gy_1))d\mu(x_1)d\mu(y_1)\leq\int \nu^{k-1}_g(B_{2.5\delta}(z))d\nu_g(z).
		\]
		If $\nu_g(B_{2\delta}(.))$ would be continuous then we would have
		\[ \int \nu^{k-1}_g(B_{2\delta}(x_1-gy_1))d\mu(x_1)d\mu(y_1)=\int \nu^{k-1}_g(B_{2\delta}(z))d\nu_g(z).
		\]
		In general, we choose a continuous function sandwiched by $\nu_g(B_{2\delta}(.))$ and $\nu_g(B_{2.5\delta}(.))$ (by taking convolution with a suitable smooth cutoff function) then apply the definition of $\nu_g$ to arrive at the above inequality.
	\end{proof}
	\section{The main result}
	In this section we give a detailed proof of Theorem \ref{HAR1}. We note that in \cite{GILP15}, a proof is given under the condition that $\nu_g$ is absolutely continuous for almost all $g\in\mathbb{O}(n)$. In \cite{Y18}, a sketched proof is given for the case when $k=3$ and we note that the same strategy works for general cases $k\geq 2$ as well and here we will provide more details.
	
	\begin{proof}[Proof of Theorem \ref{HAR1}]
		By Lemma \ref{GILP} and \ref{ENERGY} we see that as $\delta \to 0,$
		\[
		\int\nu^2_\delta(z)dz\leq C \delta^{-n(k-1)}\int \int \nu_g^{k-1}(B_{2.5\delta}(z))d\nu_g dg,
		\]
		where $C>0$ is a constant. The situation would be simple if $\nu_g(B_{2.5\delta}(z))$ would be continuous with respect to $z.$ However, we can not assume this continuity condition. To deal with this issue, let $\phi^{DD}(.)$ be a radial Schwartz function such that $\hat{\phi}^{DD}$ is real valued, non-negative,  vanishes outside the ball of radius $0.5c''>0$ around the origin and is equal to a positive number $c>0$ on a ball of radius $c'>0$ around the origin. Now we take the square $\phi^{D}=(\phi^{DD})^2$ and see that
		\[
		\hat{\phi}^{D}=\hat{\phi}^{DD}*\hat{\phi}^{DD}.
		\]
		We see that $\hat{\phi}^{D}$ is real valued, non-negative, vanishes outside the ball of radius $c''$ around the origin. Unlike $\hat{\phi}^{DD}$, $\hat{\phi}^{D}$ is no longer a constant function on any ball centred at the origin. By further rescaling if necessary, we may assume that $\phi^D(x)\geq 1$ for $x\in B_{2.5}(0).$ This can be done because $\phi^D$ is real valued, Schwartz and $\phi^D(0)>0.$ Since $\hat{\phi}^D$ is compactly supported, we can denote $c''''=\|\hat{\phi}^D\|_\infty.$ Then we write $h_{g,\delta}=\nu_g*\phi^D(\delta^{-1}.).$ We see that
		\[
		\nu_g(B_{2.5\delta}(z))=\int_{B_{2.5\delta}(z)}d\nu_g(x)\leq \int \phi^D((z-x)/\delta)d\nu_g(x)=h_{g,\delta}(z).
		\]
		Now we write $f_{g,\delta}(.)=\delta^{-n}h_{g,\delta}(.),$ as a result we see that
		\[
		\int\nu^2_\delta(z)dz\lesssim \int\int f^{k-1}_{g,\delta}(z)d\nu_g(z)dg.
		\]

		Let $\psi$ be a smooth cutoff function supported in $\{\omega\in\mathbb{R}^n: |\omega|\in [0.5,4] \}$ and identically equal to $1$ in $\{\omega\in\mathbb{R}^n: |\omega|\in [1,2] \}.$ We can also require that $\sum_{j\in\mathbb{Z}} \psi(2^{-j} \omega)=1$ and this is the starting point of the Littlewood-Paley decomposition. Let $f_{g,\delta,j}, \nu_{g,j}$ be the $j$-th Littlewood-Paley piece of $f_{g,\delta},\nu_{g}$ respectively, namely, $\hat{f}_{g,\delta,j}(\omega)=\hat{f}_{g,\delta}(\omega)\psi(2^{-j}\omega)$ and similarly for $\nu_{g,j}.$ We need to bound $\|f_{g,\delta,j}\|_{\infty}$ as well as $\|\nu_{g,j}\|_{\infty}.$ The later can be bounded by
		$
		C'2^{j(n-s)}
		$
		for any $s<\Haus F$ with a constant $C'$ depending on the function $\psi$. This was shown in \cite[page 805]{GILP15}. For the former, we will be interested in estimating $\|f_{g,\delta,j}\|_\infty$ when $2^j$ is not as large as $\delta^{-1}.$ In this case, recall that $f_{g,\delta}=\nu_g*\phi^D_\delta$ and in terms of Fourier transform we have
		\[
		\hat{f}_{g,\delta,j}=\hat{\nu}_g \hat{\phi^D_\delta} \psi(2^{-j}.)
		\]
		Recall that $\phi^D_\delta(.)=\delta^{-n}\phi^D(./\delta)$, therefore we have $\hat{\phi^D_\delta}(.)=\hat{\phi^D}(\delta.).$ Then we see that 
		\begin{eqnarray*}
			\|f_{g,\delta,j}\|_\infty&\leq& \|\hat{f}_{g,\delta,j}\|_1\\
			&\leq& c'''' \int |\hat{\nu}_g(\omega)  \psi(2^{-j}\omega)|d\omega\\
			&\leq& c''''\int_{B_{2^{j+2}}(0)} |\hat{\mu}(\omega)\hat{\mu}(g\omega)|d\omega\\
			&\leq& c''''\sqrt{\int_{B_{2^{j+2}}(0)} |\hat{\mu}(\omega)|^2d\omega \int_{B_{2^{j+2}}(0)} |\hat{\mu}(g\omega)|^2d\omega}.
		\end{eqnarray*}
		By the discussion in \cite[Section 3.8]{Ma2} we see that
		\[
		\int_{B_{2^{j+1}}(0)} |\hat{\mu}(\omega)|^2d\omega\lesssim 2^{(j+2)(n-s)}.
		\]
		The same estimate holds for $\int_{B_{2^{j+2}}(0)} |\hat{\mu}(g\omega)|^2d\omega$ as well. Therefore we see that
		\[
		\|f_{g,\delta,j}\|_\infty\leq C'2^{j(n-s)}
		\]
		where $C'>0$ is a constant which does not depend on $g,j,\delta.$ Observe that if $2^{j-1}>c''\delta^{-1}$ then $f_{g,\delta,j}=0$ and this is the reason for considering $2^j$ to be not much larger than $\delta^{-1}.$ Thus, we have obtained a complete estimate for $\|f_{g,\delta,j}\|_\infty.$

		In what follows we want to estimate the following integral,
		\[
		\int f^{k-1}_{g,\delta}(z)d\nu_g(z)\tag{*}.
		\]
		We want to apply the argument in \cite[Section 3]{GILP15} and we provide details depending on whether $k=2$ or $k\geq 3.$ We note here that the argument in \cite[Section 3]{GILP15} works only for $k\geq 3$ but we shall extend it to the case when $k=2.$ 
		
		\subsection{Case $k=2$:} In this situation, the equation $(*)$ can be written as
		\[
		\int f_{g,\delta}d\nu_g.
		\]
		We can apply \cite[Formula (3.27)]{Ma2} and as a result we see that
		\[
		\int f_{g,\delta}(z)d\nu_g(z)=\int \hat{\nu}_g(\omega) \overline{\hat{\overline{f}}}_{g,\delta}(\omega)d\omega.
		\]
		We note here that $\overline{\hat{\overline{f}}}_{g,\delta}(\omega)=\hat{f}_{g,\delta}(-\omega).$ 
		Therefore we see that
		\[
		\int\int f_{g,\delta}(z)d\nu_g(z)dg= \sum_{j\in\mathbb{Z}}\int\int \hat{f}_{g,\delta}(-\omega)\hat{\nu}_{g}(\omega)\psi(2^{-j}\omega)d\omega dg.
		\]
		Recall that $f_{g,\delta}=\nu_g* \phi^D_\delta$ we see that
		\[
		\hat{f}_{g,\delta}=\hat{\nu}_g \hat{\phi^D_\delta}.
		\]
		Then since $\nu_g=\mu*g\mu$ we see that
		\[
		\hat{\nu}_g(\omega)=\hat{\mu}(\omega)\hat{\mu}(g\omega).
		\]
		As a result we see that
		\[
		\int\int \hat{f}_{g,\delta}(-\omega)\hat{\nu}_{g}(\omega)\psi(2^{-j}\omega)d\omega dg=\int\int |\hat{\mu}(\omega)|^2|\hat{\mu}(g\omega)|^2 \hat{\phi^D_\delta}(\omega)\psi(2^{-j}\omega)d\omega dg.
		\]
		Observe that $\hat{\phi^D_\delta}$ is a cutoff function at scale $\delta^{-1}.$ More precisely, for $|\omega|>c''\delta^{-1}$ we have
		\[
		\hat{\phi^D_\delta}(\omega)=0.
		\]
		By integrating first with respect to $dg$ and then $d\omega$ we see that
		\[
		\int\int |\hat{\mu}(\omega)|^2|\hat{\mu}(g\omega)|^2 \hat{\phi^D_\delta}(\omega)\psi(2^{-j}\omega)d\omega dg=C(n) \int \left(\int_{S^{n-1}} |\hat{\mu}(t\sigma)|^2d\sigma\right)^2 \psi(2^{-j}t)\hat{\phi}^D_\delta(t) t^{n-1}  dt,\tag{**}
		\]
		where $d\sigma$ is the Lebesgue probability measure on $S^{n-1}$. We write $\hat{\phi^D_\delta}(t)=\hat{\phi^D_\delta}(\omega)$ for $|\omega|=t$ and similarly for $\psi(2^{-j}t).$ Since $\psi$ and $\phi^D_\delta$ are radial functions, the above step is well-defined. The constant $C(n)$ is a positive number which depends only on $n$.
		
		We sill need to sum $(**)$ over $j\in\mathbb{Z}.$ Because of the cutoff property of $\phi^D_\delta$ we only need to consider the sum up to $\sum_{j: 2^j\leq 2c''\delta^{-1}}.$ More precisely, there is a positive constant $C''>0$ and we have
		\[
		\int\int f_{g,\delta}(z)d\nu_g(z)dg\leq C'' \sum_{j: 2^j\leq 2c''\delta^{-1}} \int \left(\int_{S^{n-1}} |\hat{\mu}(t\sigma)|^2d\sigma\right)^2 \psi(2^{-j}t) t^{n-1}  dt.
		\]
		In fact when $2^j>2c''\delta^{-1}$ then $\psi(2^{-j}\omega)\hat{\phi^D_\delta}(\omega)=0.$ Therefore we do not need to sum larger values of $j.$ This is because we can choose a special cutoff function $\phi^D$ whose Fourier transform is compactly supported. This makes $\phi^D$ not compactly supported but we do not need this. We still need to sum negative values of $j$ but as $\mu$ is a probability measure we have
		\[
		\sum_{j\leq 0}\int \left(\int_{S^{n-1}} |\hat{\mu}(t\sigma)|^2d\sigma\right)^2 \psi(2^{-j}t) t^{n-1}  dt\leq C''' \int_{[0,2]} t^{n-1} dt<\infty,
		\]
		for a positive constant $C'''>0.$ We can use Theorem \ref{WE}(Wolff-Erdogan). For all $\epsilon>0$ there is a constant $C_\epsilon>0$ such that for each $t>0$ we have
		\[
		\int_{S^{n-1}} |\hat{\mu}(t\sigma)|^2d\sigma\leq C_\epsilon t^{-\gamma_s+\epsilon}.
		\]
		We can insert one of the factor $\int |\hat{\mu}(t\sigma)|^2d\sigma$ into $(**)$ and we see that for each $\epsilon>0,$
		\[
		\int \left(\int_{S^{n-1}} |\hat{\mu}(t\sigma)|^2d\sigma\right)^2 \psi(2^{-j}t) t^{n-1}  dt\lesssim \int_{2^{j-1}}^{2^{j+2}}\left(\int |\hat{\mu}(t\sigma)|^2d\sigma\right) t^{n-1}t^{-\gamma_s+\epsilon}dt.
		\]
		Then we see that
		\[
		\sum_{j:1\leq 2^{j}\leq 2c''\delta^{-1}}\int \left(\int_{S^{n-1}} |\hat{\mu}(t\sigma)|^2d\sigma\right)^2 \psi(2^{-j}t) t^{n-1}  dt\lesssim \int_{0}^{4c''\delta^{-1}} \left(\int |\hat{\mu}(t\sigma)|^2d\sigma\right) t^{n-1}t^{-\gamma_s+\epsilon}dt.
		\]
		Up to a multiple constant the RHS above is equal to
		\[
		\int_{|\omega|\leq 4c''\delta^{-1}} |\hat{\mu}(\omega)|^2 |\omega|^{-\gamma_s+\epsilon} d\omega.
		\]
		If $I_{n-\gamma_s-\epsilon}(\mu)<\infty$ then the above integral is bounded uniformly for $\delta\to 0,$ in this case $D(F)$ would have positive Lebesgue measure. Therefore we consider the case when $n-\gamma_s-s+\epsilon>0.$ By the discussions in Section \ref{L2} we see that
		\[
		\int_{|\omega|\leq 4c''\delta^{-1}} |\hat{\mu}(\omega)|^2 |\omega|^{-\gamma_s+\epsilon} d\omega\lesssim \delta^{-\dim_{L^2,\gamma_s+\epsilon}\mu}\leq \delta^{-(n-\gamma_s-s+\epsilon)}.
		\]
		For the rightmost inequality we need the fact that $\mu$ is an $s$-Frostman measure. Thus, we showed that
		\[
		\int \nu^2_{\delta}(z)dz\lesssim \int\int f_{g,\delta}(z)d\nu_g(z)dg\lesssim \delta^{-(n-\gamma_s-s+\epsilon)}.
		\]
		This concludes the case when $k=2.$
		
		\subsection{Case $k\geq 3$:} We need to estimate the following integral
		\[
		\int f^{k-1}_{g,\delta}(z)d\nu_g(z).
		\]    
		We see that
		\begin{eqnarray*}
			\int f^{k-1}_{g,\delta}(z)d\nu_g(z)&\overset{\text{\cite[Formula (3.27)]{Ma2}}}{=}& \int \hat{\nu}_g(\omega) \overline{\hat{\overline{f^{k-1}}}}_{g,\delta}(\omega)d\omega\\
			&=& \int(\hat{f}_{g,\delta}\overset{(k-1)-\text{times}}{*\dots*}\hat{f}_{g,\delta})(-\omega) \hat{\nu}_{g}(\omega)d\omega
		\end{eqnarray*}
		We write the Littlewood-Paley decompositions $\hat{\nu}_{g}=\sum_{j\in\mathbb{Z}}\hat{\nu}_{g,j}$ and $\hat{f}_{g,\delta}=\sum_{j\in\mathbb{Z}}\hat{f}_{g,\delta,j}.$ Then we see that
		\[
		\int(\hat{f}_{g,\delta}\overset{(k-1)-\text{times}}{*\dots*}\hat{f}_{g,\delta})(-\omega) \hat{\nu}_{g}(\omega)d\omega=\sum_{j_1,j_2,\dots,j_k}\int(\hat{f}_{g,\delta,j_1}\overset{(k-1)-\text{times}}{*\dots*}\hat{f}_{g,\delta,j_{k-1}})(-\omega) \hat{\nu}_{g,j_k}(\omega)d\omega.
		\]
		Denote $j^*=\max\{j_1,\dots,j_{k-1}\}$. We see that $\hat{f}_{g,\delta,j_1}\overset{(k-1)-\text{times}}{*\dots*}\hat{f}_{g,\delta,j_{k-1}}$ is supported on an annulus. We can estimate the location of this annulus. First, each term of form $\hat{f}_{g,\delta,j}$ is supported on an annulus with inner radius $2^{j-1}$ and outer radius $2^{j+2}.$ Thus $\hat{f}_{g,\delta,j_1}\overset{(k-1)-\text{times}}{*\dots*}\hat{f}_{g,\delta,j_{k-1}}$ is supported on an annulus with inner radius at least $2^{j^*-1}$ and outer radius at most $(k-1)2^{j^*+2}.$ Thus, if either $2^{j_1+2}<2^{j^*-1}$ or $2^{j_1-1}> (k-1)2^{j^*+2}$ we see that
		\[
		\int f_{g,\delta,j_1}(z)\dots f_{g,\delta,j_{k-1}}(z)\nu_{g,j_k}(z)dz=\int(\hat{f}_{g,\delta,j_1}\overset{(k-1)-\text{times}}{*\dots*}\hat{f}_{g,\delta,j_{k-1}})(-\omega) \hat{\nu}_{g,j_k}(\omega)d\omega=0.
		\]
		For this reason we only need to sum the terms indexed by $j_1,\dots,j_k$ with $|j^*-j_k|\leq C(k)$ for a constant $C(k)$ depending only on $k.$ Let $j$ be any integer and we sum all the terms $j_1,\dots,j_k$ with $|j_1-j|\leq C(k)/2$ and $|j^*-j|\leq C(k)/2.$ The resulting sum is bounded from above by a constant (depending on $k$) times the following expression
		\[
		2^{j(n-s)(k-2)}\int \left|\sum_{|q|\leq C(k)/2} f_{g,\delta,j+q}(z)\right|\left|\sum_{|q|\leq C(k)/2} \nu_{g,j+q}(z)\right|dz.\tag{***}
		\]
		We need to sum the above expression for $j\in\mathbb{Z}.$ If $2^{j-C(k)/2}\leq 2c''\delta^{-1}$, then $(***)$ is bounded from above by
		\[
		c''''2^{j(n-s)(k-2)}\int \left|\sum_{|q|\leq C(k)/2} \nu_{g,j+q}(z)\right|^2dz.
		\]
		Here we used Cauchy-Schwartz inequality, Plancherel's theorem as well as the fact that $\|\hat{\phi}^D\|_\infty=c''''.$
		If $2^{j-C(k)/2}>2c''\delta^{-1}$, then $(***)$ is equal to $0.$
		In all, the sum for $j\in\mathbb{Z}$ of $(***)$ can be bounded from above by
		\[
		c''''\sum_{2^j\leq 2c''\delta^{-1}} 2^{j(n-s)(k-2)}\int \left|\sum_{|q|\leq C(k)/2} \nu_{g,j+q}(z)\right|^2dz.
		\]
		It is easy to check that the sum with $j\leq 0$ gives another constant $C(k,s,\nu)$ depending on $k,s$ and $\nu$. Then we can summarize our results so far in the following inequality,
		\[
		\int f^{k-1}_{g,\delta}(z)d\nu_g(z)\lesssim \sum_{1\leq 2^j\leq 2c''\delta^{-1}} 2^{j(n-s)(k-2)}\int \left| \tilde{\nu}_{g,j}(z)\right|^2dz+C(k,s,\nu),
		\]
		where we have written $\tilde{\nu}_{g,j}=\sum_{q} \nu_{g,j+q}$ for simplicity. The functions $\nu_{g,j}$ are real valued for all $j\in\mathbb{Z}$ because $\nu_g$ is a real valued measure and $\psi$ is a radial function. Then we see that
		\[
		\int \left| \tilde{\nu}_{g,j}(z)\right|^2dz=\int |\hat{\tilde{\nu}}_{g,j}(\omega)|^2d\omega.
		\]
		Recall that $\nu_g=\mu*g\mu,$ we see that
		\[
		\int |\hat{\tilde{\nu}}_{g,j}(\omega)|^2d\omega\leq \int_{|\omega|\in [2^{j-C(k)/2-1},2^{j+C(k)/2+2}]} |\hat{\mu}(\omega)|^2|\hat{\mu}(g\omega)|^2 d\omega.
		\]
		The integral against $dg$ of the RHS above is a constant multiply 
		\[
		\int_{2^{j-C(k)/2-1}}^{2^{j+C(k)/2+2}} \left(\int |\hat{\mu}(t\sigma)|^2d\sigma\right)^2 t^{n-1}dt\lesssim 
		\int_{2^{j-C(k)/2-1}}^{2^{j+C(k)/2+2}} \left(\int |\hat{\mu}(t\sigma)|^2d\sigma\right) t^{n-1} t^{-\gamma_s+\epsilon}dt,
		\]
		where the above inequality holds for all $\epsilon>0.$ As in the case when $k=2$, we see that,
		\[
		\int\int f^{k-1}_{g,\delta}(z)d\nu_g(z) dg\lesssim C(k,s,\nu)+\sum_{j: 1\leq 2^j\leq 2c''\delta^{-1}} 2^{-j(\gamma_s-\epsilon)}2^{j(n-s)}2^{j(n-s)(k-2)}.
		\]    
		If $(n-s)(k-1)-\gamma_s+\epsilon<0$ then $\nu$ would be an $L^2$ function, otherwise we see that
		\[
		\int \nu^2_\delta(z)dz\lesssim  \delta^{-((n-s)(k-1)-\gamma_s+\epsilon)}.
		\]
		This concludes the proof for the case when $k\geq 3.$
	\end{proof}

	\section{Asymmetric distance sets}\label{asy}
	Let $n\geq 2$ be an integer. Let $F_1,F_2$ are two Borel sets in $\mathbb{R}^n$ with $\Haus F_1=s_1, \Haus F_2=s_2.$  Let $\mu_1,\mu_2$ be probability measures supported on $F_1,F_2$ respectively. For $g\in\mathbb{O}(n)$, the orthogonal group on $\mathbb{R}^n$, we construct a measure $\nu_g$ as follows,
	\[
	\int_{\mathbb{R}^n} f(z)d\nu_g(z)=\int_{F_1}\int_{F_2} f(u-gv)d\mu_1(u)d\mu_2(v),f\in C_0(\mathbb{R}^n).
	\]
	In other words, $\nu_g=\mu_1*g\mu_2.$ We also construct a measure $\nu$ by
	\[
	\int f(t)d\nu(t)=\int f(|x_1-x_2|)d\mu_1(x_1)d\mu_2(x_2).
	\]
	It can be seen that $\nu$ is supported on
	\[
	D(F_1,F_2)=\{|x_1-x_2|: x_1\in F_1, x_2\in F_2\}.
	\]
	Most of the argument in previous sections can be used here. In particular, one can show that for each $\epsilon>0$
	\[
	\|\nu_\delta\|_2^2\lesssim \delta^{-(n-\gamma_{s_1}-s_2-\epsilon)}
	\]
	and
	\[
	\|\nu_\delta\|_2^2\lesssim \delta^{-(n-\gamma_{s_2}-s_1-\epsilon)}.
	\]
	Therefore we see that if $\max\{\gamma_{s_1}+s_2,\gamma_{s_2}+s_1\}>n$ then $D(F_1,F_2)$ has positive Lebesgue measure. If $s2\geq s_1$ then this is equivalent to $s_2+0.5s_1>0.75n+0.5.$ Now we turn to Corollary \ref{ma}. With the same arguments as above we see that if $s_2+0.5s_1>0.75n+0.5$, then for almost all $g\in\mathbb{O}(n)$, $\nu_g$ is absolutely continuous with respect to the Lebesgue measure. In general, we can consider $k\geq 3$ and obtain conditions for $\nu_g$ to be $L^{k}$ for almost all $g\in\mathbb{O}(n).$
	\section{Further discussions: Tubular incidence estimates}\label{dis}
	We discuss here tubular incidence estimates and their connections with Falconer's distance set problem and the Kakeya problem. Before the discussions, we note that tubular situations are often quite different (and perhaps more difficult) than their discrete analogies. See for example \cite[Remark 1.5]{W}. The proofs of discrete point-line incidence estimates in $\mathbb{R}^3$, for example Szemeredi-Trotter's and Guth-Katz's bound, often involve graph theoretic or polynomial methods. Those methods can not be directly applied in tubular settings. Nonetheless, we will use the discrete incidence estimates as intuitive guesses for what could be done in tubular settings.
	\subsection{Falconer's distance set problem and tubular incidence estimates}
	In this section, we discuss a different group-theoretic approach to distance set ($k=2$). Hopefully, this sheds some lights on the Falconer's distance problem and other similar problems.
	
	Let $\mu_1,\mu_2\in\mathcal{P}([0,1]^2)$ be  Borel probability measures. Let $S_1,S_2\in\mathbb{O}(2)$ and $s_1,s_2\in\mathbb{R}^2$ then we see that for each $x\in\mathbb{R}^2$
	\[
	S_2(S_1(x)+s_1)+s+2= S_2S_1(x)+S_2(s_1)+s_2,
	\]
	and
	\[
	S_1(S_1^{-1}(x)-S_1^{-1}(s_1))+s_1=x_1.
	\]
	In this way we see that actions $S(.)+s$ for $S\in\mathbb{O}(2),s\in\mathbb{R}^2$ form a group. We denote this group as $\mathbb{E}(2).$ We can equip $\mathbb{E}(2)$ with the product topology $\mathbb{O}(2)\times \mathbb{R}^2$. Then $E(2)$ is a locally compact topological group. Since we are interested in $[0,1]^2$, we only need to consider a compact subset $U(2)$ (not a subgroup) of $\mathbb{E}(2)$ whose $\mathbb{R}^2$ part is bounded inside the ball of radius $3$ centred at the origin. We can equip $\mathbb{E}(2)$ with the metric in $\mathbb{O}(2)\times\mathbb{R}^2.$ The Haar measure $\lambda$ which gives measure $1$ on $U(2)$ is equivalent to the Lebesgue on $\mathbb{O}(2)\times\mathbb{R}^2.$ More precisely $\lambda=c\lambda_{\mathbb{O}(2)}\times \lambda_{\mathbb{R}^2}$, where $c>0$ is a constant, $\lambda_{O(2)}$ is the Haar probability measure of $\mathbb{O}(2)$ and $\lambda_{\mathbb{R}^2}$ is the Lebesgue measure.
	
	Let $\nu$ be the asymmetric distance measure discussed in the previous section. It can be checked that
	\[
	\|\nu_\delta\|_2^2\lesssim \delta^{-1}\mu_1^{2}\times \mu_2^2\{((x_1,x_2),(x_3,x_4): ||x_1-x_2|-|x_3-x_4||\leq \delta \}.
	\]
	For each $x_1,x_2,x_3,x_4$ with $||x_1-x_2|-|x_3-x_4||\leq \delta$ we can find $S\in\mathbb{O}(2)$ and $s\in\mathbb{R}^2$ such that
	\[
	|S(x_1)+s-x_3|\leq 2\delta,|S(x_2)+s-x_4|\leq 2\delta. 
	\]
	Now we cover $U(2)$ with $N(\delta)$ many $\delta$-balls and choose for each $i\in\{1,\dots,N(\delta)\}$ an element $g_i$ in each $\delta$-ball. Then we see that
	\[
	\{x_1,x_2,x_3,x_4: ||x_1-x_2|-|x_3-x_4||\leq \delta \}\subset \bigcup_{g_i}\{x_1,x_2,x_3,x_4: |x_1-g_ix_3|\leq 2\delta,|x_2-g_ix_4|\leq 2\delta\}.
	\]
	Then similar argument as in \cite[Section 2]{GILP15} shows that
	\begin{eqnarray*}
		& &\delta^{-1}\mu_1^{2}\times \mu_2^2\{((x_1,x_2),(x_3,x_4): ||x_1-x_2|-|x_3-x_4||\leq \delta \}\\
		&\lesssim& \delta^{-4} \int_{U(2)} \mu_1^2\times\mu^2_2\{((x_1,x_2),(x_3,x_4)): |x_1-g(x_3)|\leq 2\delta, |x_2-g(x_4)|\leq 2\delta\}  dg\\
		&=&  \delta^{-4} \int_{U(2)} (\mu_1\times\mu_2\{x_1,x_3: |x_1-g(x_3)|\leq 2\delta\})^2  dg.
	\end{eqnarray*}
	
	As we will see later, we need to consider the elements in $U(2)$ whose $\mathbb{O}(2)$ part has small rotation angle. We define $U'(2)$ to be the elements in $U(2)$ whose $\mathbb{O}(2)$ part $S$ is away from the identity, more precisely, $\|S-I\|\geq 0.1$ in operator norm. We assume that
	\[
	\int_{U'(2)} (\mu_1\times\mu_2\{x_1,x_3: |x_1-g(x_3)|\leq 2\delta\})^2  dg\geq 0.5 \int_{U(2)} (\mu_1\times\mu_2\{x_1,x_3: |x_1-g(x_3)|\leq 2\delta\})^2  dg.\tag{Tech}
	\]
	The intuition behind the above condition is that $\mu_1,\mu_2$ do not prefer any particular direction. An extreme case which does not satisfy the above condition is when $\mu_1$ and $\mu_2$ are supported on finitely many line segments. We need condition $(Tech)$ for a particular coordinate system which will be introduced later and it is very likely to be redundant.
	
	For each $j\in\mathbb{Z}$ we consider the set
	\[
	G_j=\{g\in U(2):   2^j\leq \mu_1\times\mu_2\{x_1,x_3: |x_1-g(x_3)|\leq 2\delta\}\leq 2^{j+1}         \}.
	\]
	Then we see that
	\[
	\int_{U(2)} (\mu_1\times\mu_2\{x_1,x_3: |x_1-g(x_3)|\leq 2\delta\})^2  dg\approx \sum_{j} 2^{2j}\lambda(G_j).
	\]
	For each $x_1,x_3\in [0,1]^2$ we denote the following tubular set
	\[
	T(x_1,x_3,\delta)=\{g\in U(2): |x_1-g(x_3)|\leq \delta     \}.
	\]
	We choose a $\delta$-separated set $F_1$ and $F_2$ in $supp(\mu_1)+B_\delta(0)$ and $supp(\mu_1)+B_\delta(0)$ respectively. If
	\[
	2^j\leq \mu_1\times \mu_2\{x_1,x_3: |x_1-g(x_3)|\leq 2\delta\}\leq 2^{j+1}  ,   
	\]
	there should be some number (say, $N_\delta(\mu_1,\mu_2,j)$) of pairs $x_1,x_3$ such that
	\[
	g\in T(x_1,x_3,\delta).
	\]
	Assume for now that $\mu_1,\mu_2$ are AD-regular with the same exponent $s\in (0,2).$ Then we can estimate $N_\delta(\mu_1,\mu_2,j)$ by $2^j\delta^{-2s}.$ Thus, we need to estimate the $\lambda$ measure of the set in $U(2)$ which lies in roughly $2^j\delta^{-2s}$ many tubular sets of form $T(x_1,x_3,\delta).$ In general, if $\mu_1,\mu_2$ are $s$-Forstman measures then we can bound $N_\delta(\mu_1,\mu_2,j)$ from below by $2^j\delta^{-2s}.$ 
	
	We have translated the distance problem into a different point of view. For each pair $(x_1,x_3)\in F_1\times F_2$ we draw the tubular set $T(x_1,x_3,\delta)\subset U(2).$ For each $j\in\mathbb{Z}$ we need to estimate the $\lambda$ measure of $2^j\delta^{-2s}$-rich set
	\[
	R_{j}=\{g\in U(2): g\text{ lies in at least $2^j\delta^{-2s}$ many tubular sets}\} .
	\]
	We expect that for large $j$, $\lambda(R_j)$ should be small. This can be seen as a tubular version of the incidence problem. It is difficult to consider the incidence problem directly in $\mathbb{O}(2)\times\mathbb{R}^2.$ We need to choose a suitable coordinate system. Let $(S,s)\in \mathbb{O}(2)\times\mathbb{R}^2.$ Then we see that if $S\neq I$, the identity in $\mathbb{O}(2),$ we can find $x_0=(S-I)^{-1}(-s).$ Then for each $x\in\mathbb{R}^2$ we have
	\[
	S(x)+s=S(x-x_0)+S(x_0)+s=S(x-x_0)+x_0-s+s=S(x-x_0)+x_0.
	\]
	We see that the action $(S,s)$ is a rotation around the fix point $x_0.$ If $S=I$ we can still intuitively consider $(S,s)$ as a rotation around the point of $\infty.$ Now we have the following map
	\[
	G_1:(S,s)\in\mathbb{O}(2)\times\mathbb{R}^2\to (S, (S-I)^{-1}(-s))\in\mathbb{O}(2)\times\mathbb{R}^2.
	\]
	In plain words, an action of form $S(.)+s$ is mapped under $G_1$ to the corresponding action defined by rotating around a fix point. Since $\mathbb{O}(2)$ is a dimension one Lie group with two connected component. Any element of $\mathbb{O}(2)$ can be written as a rotation around the origin composed  possibly a mirror reflection with respect to a line through the origin. We can identify the rotation part $\mathbb{O}(2)$ with $[0,2\pi].$ Denote $\theta$ to be this identification map. Then for each $(S,s)\in\mathbb{O}(2)\times\mathbb{R}^2$ we can give it the coordinate $(s,\cot (\theta(S)/2))\in\mathbb{R}^3.$ Denote this coordinate function as $G_2,$ then $G_2 \circ G_1$ is the coordinate function we want to use for $\mathbb{E}(2).$ Under this coordinate system the tubular sets of form $T(x_1,x_3,\delta)$ are essentially $\delta$-neighbourhoods of lines. We need to consider elements in $U(2)$ with trivial $\mathbb{O}(2)$ parts (pure translations) separately since such elements are mapped to $\mathbb{R}^2\times\{\infty\}$ under $G_2\circ G_1.$ Similar considerations can be found in \cite[Section 9.4]{Guth}. We consider the standard topology and the Lebesgue measure $\lambda_{\mathbb{R}^3}$ on $\mathbb{R}^3,$ if we pull back those structures via $(G_2 \circ G_1)^{-1}$ the resulting structures are different than the original topology and measure we chose for $\mathbb{E}(2).$ However, the map $G_2\circ G_1$ is not contracting in the sense that the image of any $r$-ball in $\mathbb{E}(2)$ contains a $c_1r$-ball in $\mathbb{R}^3$ where $c_1>0$ is a constant. On the other hand, we can see that $G_2\circ G_1$ is expands a lot if the rotation part is close to the identity. By the above arguments we see that there is a constant $c_2>0$ such that for any measurable set $A\subset\mathbb{E}(2)$, we have
	\[
	c_2^{-1}(\min\{\theta: (x_1,x_2,\theta)\in G_2\circ G_1 (A)\})^{4} \lambda_{\mathbb{R}^3}(G_2\circ G_1 (A))\leq \lambda(A)\leq c_2 \lambda_{\mathbb{R}^3}(G_2\circ G_1 (A)).
	\]
	Since we have the technical condition $(Tech)$ we only need to consider $\theta>0.1$. Therefore, the push back of the standard topology and measure structure from $\mathbb{R}^3$ to $\mathbb{E}(2)$ is equivalent to the original structures we introduced. For simplicity we do not write $G_2\circ G_1$ since we will be always working with the new coordinate system. For the same reason we use $\lambda$ for the Lebesgue measure on $\mathbb{R}^3.$
	Suppose that two tubes $T(x_1,x_3,\delta)\cap T(y_1,y_3,\delta)\neq\emptyset,$ then there is a $g$ such that $x_1$ is $\delta$-close to $g(x_3)$ and $y_1$ is $\delta$-close to $g(y_3).$ If $g$ is not a translation, then it is a rotation with a fixed point. Since $F_1, F_2$ are $\delta$-separated, if the distance between $F_1$ and $F_2$ is larger than $0.5$, we can assume that $T(x_1,x_3,\delta), T(y_1,y_3,\delta)$ are pointing at directions with separation at least $0.01\delta$ (viewing as tubes in the new coordinate system). This separation conditions does not hold in general, here the $0.5$-separation condition between $F_1,F_2$ is crucial. We write the cardinality of $F_1, F_2$ as $\delta^{-s_B}$ (assume for simplicity that $F_1,F_2$ have the same cardinality), where $s_B$ can be essentially interpreted as the box dimension of $supp(\mu_1)$ or $supp(\mu_2).$ Motivated by a result in \cite{GK} we expect (although it might be really difficult to show) that \[\lambda(R_j)\lesssim (\delta^{-2s_B})^{1.5}(2^j\delta^{-2s})^{-2}\delta^{3}=2^{-2j}\delta^{3+4s-3s_B}.\]
	If $\mu$ is lower regular in the sense that there is $\alpha>0$ such that $\mu_1(B)\geq 0.5\delta^{\alpha},\mu_2(B)\geq 0.5\delta^{\alpha}$ for all $\delta$ and the above estimate would be true then we see that 
	\[
	\|\nu_\delta\|_2^2\lesssim \delta^{-4}\sum_{j} 2^{2j}\lambda(G_j)\lesssim \delta^{-4}\sum_{j} 2^{2j}\lambda(R_j)\lesssim \max\{-\delta^{-1+4s-3s_B}\log \delta,1\}.
	\]
	Furthermore, if $4s-3s_B>1$ then $\nu$ would have $L^2$ density. Here we observe that there are only roughly $-\log \delta$ many values of $j$ which need to be considered. This is because $\mu_1\times \mu_2\{x_1,x_3: |x_1-g(x_3)|\leq 2\delta\}\in [0.25\delta^{-2\alpha},100\delta^s].$ The lower bound comes from the lower regularity and the upper bound comes from the fact that $\mu$ is a $s$-Frostman measure. We pose here the following problem.
	\begin{ques}\label{hard}
		Let $l_i,i\in\{1,\dots,L\}$ be $L$ unit segments in $\mathbb{R}^3.$ Suppose that each degree $2$ algebraic surface contains at most $\sqrt{L}$ many of those unit segments. Let $\delta>0$ be a small number. Suppose further that if $L_i^\delta\cap l_j^\delta\neq\emptyset$ then the directions of $l_i,l_j$ separate at least $\delta.$ For each $r\geq 1,$ let $P_r$ be the set of points contained in at least $r$ many of tubes $l^\delta_i,i\in\{1,\dots,L\}.$ Can the following inequality holds with a suitable constant $C>0,$\[
		\lambda(P_r)\leq C \delta^{3}(L^{1.5}r^{-2}+L/r)?\tag{Guess}
		\]
		Here we can assume that $\delta^{-\alpha}\leq L\leq 2\delta^{-\alpha}$ for a number $\alpha\in [2,3).$
	\end{ques}
	The above result is a direct analogy of a discrete incidence estimate in $\mathbb{R}^3$ which was proved in \cite{GK} with polynomial methods. Notice that we have ignored the effect of the term $L/r$. In fact, this term is significant if $r>\sqrt{L}.$ In our setting, this is $2^j\geq \delta^{s}$ and therefore we can safely ignore this term. Unfortunately, $(Guess)$ seems not to be easily obtainable directly by alternating the proofs of their discrete relatives. However, we shall provide an estimate which is much weaker than $(Guess).$
	
	If $F$ is AD-regular with dimension $s$. We consider $F=F_1\cup F_2$ as discussed before with $F_1,F_2$ separated by at least $0.5.$ Let $\mu_1,\mu_2$ be AD-regular measures support on $F_1,F_2$ respectively. Assume that the condition $(Tech)$ is satisfied. Let $\delta>0$ be a small number and we choose $\delta$-separated subsets $F_{1,\delta}, F_{2,\delta}$ of $F_1,F_2$ respectively. We are interested in the tubular sets
	\[
	T(x_1,x_3,\delta), x_1\in F_{1,\delta}, x_3\in F_{2,\delta}.
	\]
	It is more convenient to consider the lines
	\[
	l_{x_1,x_3}=T(x_1,x_3,0), x_1\in F_{1,\delta}, x_3\in F_{2,\delta}
	\]
	and their $\delta$-neighbourhoods. By the argument in the proof of Theorem \ref{weak}, we need to estimate the following sum
	\[
	\sum_{
		(x_1,y_1,x_3,y_3)\in F^2_{1,\delta}\times F^2_{2,\delta}, (x_1,x_3)\neq (y_1,y_3)}\lambda(l_{x_1,x_3}^{3\delta}\cap l_{y_1,y_3}^{3\delta}).
	\]
	Let $(x_1,x_3)\in F_{1,\delta}\times F_{2,\delta}$ be fixed. Let $\Delta>0$ be a number which is larger than $\delta$. We want to estimate the number of lines of form $l_{y_1,y_3}$ such that $l^{3\delta}_{x_1,x_3}\cap l^{3\delta}_{y_1,y_3}\neq \emptyset$ and the directions between $l_{x_1,x_3}, l_{y_1,y_3}$ are bounded by $\Delta$ and $2\Delta.$ The direction vector of $l_{x_1,x_3}$ is $((x_1+x_3)/|x_1+x_3|, 0.5|x_1-x_3|).$ If $l_{z_1,z_3}\cap l_{x_1,x_3}\neq\emptyset$ then there is a vector $x_0$ such that
	\[
	\angle x_1x_0x_3=\angle z_1x_0z_3.
	\]
	The separation condition between $F_{1,\delta}, F_{2,\delta}$ ensures that $|x_1-x_3|\geq 0.5.$ If the directions of $l_{x_1,x_3}$ and $l_{z_1,z_3}$ are separated by at most $\Delta,$ then the directions of $x_0-(x_1+x_2)/2$ and $x_0-(z_1+z_3)/2$ are separated by at most $\Delta$ and at the same time, $|z_1-z_3|$ and $|x_1-x_3|$ differs by at most $\Delta.$ Since the angle $\angle x_1x_0x_3$ is larger than $1/10,$ we see that $z_1,z_3$ must be contained in a $10000\Delta$-ball around $(x_1+x_3)/2.$ There are $\lesssim (\Delta/\delta)^{2s}$ many such pairs $z_1\in F_{1,\delta}, z_3\in F_{2,\delta}.$ By summing up the above bounds for $\Delta=2^j\delta$ with integers $j\geq 1$, we have the following estimate
	\[
	\sum_{
		(x_1,y_1,x_3,y_3)\in F^2_{1,\delta}\times F^2_{2,\delta}, (x_1,x_3)\neq (y_1,y_3)}\lambda(l_{x_1,x_3}^{3\delta}\cap l_{y_1,y_3}^{3\delta})\lesssim L\sum_{j\in\mathbb{N}:2^j\leq 2\delta^{-1}} 2^{2sj}\delta^3/(2^j \delta)\lesssim \delta^{3-2s}L.
	\]
	Recall that we have $L\approx \delta^{-2s}$ since $F$ is AD-regular. In this case we have
	\[
	\lambda(P_r)\lesssim \delta^3 L^2/r^2.\tag{Weak}
	\]
	
	This bound is much weaker than $(Guess).$ As the counting argument above is all very elementary, it is possible that one can improve the above result $(Weak)$, for example, the following Szemer\'{e}di-Trotter type bound
	\[
	\lambda(P_r)\lesssim \delta^3 L^2/r^3.\tag{SzT}
	\]
	If this is the case we see that
	\[
	\|\nu_\delta\|_2^2\lesssim \delta^{-4}\sum_{j}2^{2j} \delta^3 \delta^{-2s}/(2^j \delta^{-2s})^3= \sum_j \delta^{-1+2s} 2^{-j}.
	\]
	Since $2^{-j}$ can be at most as large as $\delta^{-s}$ we see that
	\[
	\|\nu_\delta\|_2^2\lesssim \max\{\delta^{-1+s},1\}.
	\]
	Therefore if $s>1$ we see that $\nu$ is absolutely continuous with respect to the Lebesgue measure. We collect the results we proved so far in the following theorem.
	
	\begin{thm}
		Let $F\subset\mathbb{R}^2$ be a Borel set with $\Haus F>s>1.$ Suppose that $F$ supports a $s$-Frostman measure $\mu$ with condition $(Tech).$  Then $(Guess)$ implies that $D(F)$ has a positive Lebesgue measure. If moreover $\mu$ is AD-regular, then the weaker estimate $(SzT)$ implies that $D(F)$ has a positive Lebesgue measure.
	\end{thm}
	\subsection{Tubular incidence estimates and the Kakeya problem}
	Tubular estimates are related to the Kakeya problem. In this section, we consider the problem in $\mathbb{R}^3.$ A Kakeya set is a subset of $\mathbb{R}^3$ which contains unit line segments along all directions. It is of great interest to show whether all Kakeya sets have full box dimension. We start this section by showing a tubular incidence estimate.
	\begin{thm}\label{weak}
		Let $\{l_i\}_{i\in\{1,\dots,L\}}$ be $L$ unit line segments in $\mathbb{R}^3$ whose the directions of are $\delta$-separated, then there is a constant $C>0$ such that
		\[
		\lambda(P_r)\leq C\delta^2 L^{1.5}/r^2.
		\]
	\end{thm}
	\begin{proof}[Sketch]
		We note here that the direction separation condition implies that $L\lesssim \delta^{-2}.$ We can decompose $\mathbb{R}^3$ into $\delta$-cubes with disjoint interiors. For each $\delta$-cube, to call it a $k$-cube if there are $k$ lines among $\{l_1,\dots,l_L\}$ intersecting it. We need to count essentially the number of $k$-cubes for $k\geq r.$ Let $k$ be an integer and consider a $k$-cube $K$. Suppose that $l_1,\dots,l_k$ are the $k$ lines intersecting $K$. We want to consider $l^{3\delta}_i\cap l^{3\delta}_j$ for $i,j\in \{1,\dots,k\}.$ We see that $\lambda(l^{3\delta}_i\cap l^{3\delta}_j\cap K)\geq c \delta^3$ where $c>0$ is an absolute constant. Since there are $\approx k^2$ many pairs of lines intersecting $K$, by summing up $k\geq 2$ we see that
		\[
		\sum_{k\geq 2}c\delta^3 k^2 \#_k\leq \sum_{1\leq i<j\leq L}\lambda(l^{3\delta}_i\cap l^{3\delta}_j),
		\] 
		where $\#_k$ is the number of $k$-cubes. Then we see that
		\[
		\lambda(P_r)\leq 1000\sum_{k\geq r} \delta^3 \#_k \leq 1000 r^{-2} \sum_{1\leq i<j\leq L}\lambda(l_i^{3\delta}\cap l_j^{3\delta}).
		\]
		Since the directions of lines are $\delta$-separated, we see that sum $\sum_{i,j}$ in above achieve the largest value when the directions between lines are as close as possible. This gives us the following estimate
		\[
		\sum_{1\leq i<j\leq L}\lambda(l_i^{3\delta}\cap l_j^{3\delta})\leq 1000 \delta^2 L^{1.5}.
		\]
		Therefore we see that
		\[
		\lambda(P_r)\leq 1000^2 \delta^2 L^{1.5}/r^2.
		\]
		This concludes the proof.
	\end{proof}
	In the above proof we use the worst case bound for $    \sum_{1\leq i<j\leq L}\lambda(l_i^{3\delta}\cap l_j^{3\delta})$ and it is simple to see that for this worst case bound to be achieved, all the tubes have to intersect in a small region. In this case we can perform the following estimate. First, we choose a maximal $10\delta$-separated set $\mathcal{S}_\delta$ in $S^2.$ For each $e\in \mathcal{S}_\delta$ we choose $l_e$ to be the line passing through the origin with direction $e$. Then we consider $l^\delta_i,i\in\mathcal{S}_\delta.$ Let $r\geq 2$ be an integer. If a point $x\in\mathbb{R}^3$ lies in at least $r$ many of the tubes then $|x|\leq 10/\sqrt{r}.$ To see this, let $|x|S^{2}$ be the sphere passing through $x.$ The intersections of $l_i,i\in\mathcal{S}_\delta$ with $|x|S^{2}$ forms a $|x|\delta$-separated set. Therefore if $|x|\geq 10/\sqrt{r}$ then the separation is at least $10\delta/\sqrt{r}.$ Then  $B_{\delta}(x)\cap |x|S^{2}$ intersects at most $0.5r$ many lines. Therefore we see that
	\[
	\lambda(P_r)\leq \lambda(B_{10/\sqrt{r}}(0))\lesssim \delta^3 L^{1.5}/r^{1.5}.
	\]
	For the last inequality we observe that $L\approx \delta^{-2}.$ It is also possible to check that the other direction holds as well
	\[
	\lambda(P_r)\gtrsim \delta^3 L^{1.5}/r^{1.5}.
	\]
	Therefore the estimate in Theorem \ref{weak} is unlikely to be sharp. 
	
	Without further conditions, estimate $(Guess)$ above cannot hold in the direction separation case. Now, let us examine a random case. Suppose choose the positive of each line segment in $\{l_i\}_{i\leq L}$ randomly. More precisely, we choose the centre of each $l_i$ according to the Lebesgue measure in $[0,1]^3.$ Then we see that the expected value of $\sum_{1\leq i<j\leq L}\lambda(l_i^{3\delta}\cap l_j^{3\delta})$ is roughly $\delta$ times the worst case bound. This is because for each pair of lines $l_i,l_j,$ the chance that the distance between $l_i,l_j$ is smaller than $3\delta$ is $\lesssim \delta.$ For example if $l_i,l_j$ are orthogonal, then for each fixed $l_i,$ the centre of $l_j$ must be inside a $8\delta$-neighbourhood of a disc of radius $10$. This gives a probability of order $\delta.$ On the other hand, if $l_i,l_j$ are almost parallel in the sense that their directions are separated by less than $10\delta$ then for each fixed $l_i,$ then centre of $l_j$ must be contained in a $40\delta$-neighbourhood of a line segment of length $10$. This gives a probability of order $\delta^2.$ For all other cases we can perform similar arguments. Thus we see that the expected value of $\lambda(P_r)$ is $\lesssim \delta^3 L^{1.5}/r^2.$
	
	It is interesting to see whether the case when all the lines are as close to each other as possible actually achieves the highest tubular incidence counting. For this reason, we pose the following problem.
	\begin{conj}\label{harder}
		Let $\{l_i\}_{i\in\{1,\dots,L\}}$ be $L$ unit line segments in $\mathbb{R}^3$ whose the directions of are $\delta$-separated, then there is a constant $C>0$ such that
		\[
		\lambda(P_r)\leq C\delta^3 L^{1.5}/r^{1.5}.
		\]
	\end{conj}
	We show that Conjecture \ref{harder} implies that any Kakeya set in $\mathbb{R}^3$ has full box dimension.
	\begin{thm}
		Assume Conjecture \ref{harder}. Under the hypothesis of Theorem \ref{weak} with $L\approx \delta^{-2}$ then 
		\[
		\lambda(\cup_i l_i^{3\delta})\geq C (-\log \delta)^{-2},
		\]
		where $C>0$ is a constant which does not depend on $\delta$ and the choices of lines $\{l_i\}_{i\leq L}.$
	\end{thm}
	\begin{proof}
		We continue using the notation in the proof of Theorem \ref{weak}. We see that
		\[
		\lambda(\cup_i l_i^{3\delta})\approx \sum_{k\geq 1} \delta^{3} \#_k.
		\]
		Now observe the followings estimate,
		\[
		\sum_{k\geq 1}\delta^{3} k\#_k\approx \sum_{i}\lambda(l_i^{3\delta}) .
		\]
		By Conjecture \ref{harder} we see that for each $r\geq 2$
		\[
		\sum_{k\geq r} \delta^{3} \#_k\leq \delta^{3} \delta^{-2\times 1.5}/r^{1.5}.
		\]
		By H\"{o}lder's inequality we see that
		\[
		\sum_{k\geq 2} k\#_k\leq \left(\sum_{k\geq 2} k^{1.5}\#_k \right)^{1/1.5}\left(\sum_{k\geq 2} \#_k \right)^{1/3}.\tag{H}
		\]
		Now we see that there is a number $C'>0$ such that
		\[
		\sum_{k\geq 2} k^{1.5}\#_k=\sum_{j\geq 1}\sum_{k\in [2^j,2^{j+1}]} k^{1.5} \#_k\leq \sum_{j\geq 1} 2^{1.5j+1.5} \delta^{-3} 2^{-1.5j}\leq C' \delta^{-3}(-\log \delta).
		\]
		For the last step we use the fact that we only need to consider $j$ such that $2^j\leq L.$ If $\#_1\geq \sum_{k\geq 2}k\#_k$ then
		\[
		\sum_{k\geq 1}\#_k\geq \#_1\geq 0.5 \sum_{k\geq 1}k\#_k\gtrsim \delta^{-3}. 
		\]
		If  $\#_1\leq \sum_{k\geq 2}k\#_k$ then by $(H)$
		\[
		\left(\sum_{k\geq 1}\#_k\right)^{1/3}\geq \left(\sum_{k\geq 2}\#_k\right)^{1/3}\gtrsim \frac{\delta^{-3}}{(\delta^{-3}(-\log \delta))^{1/1.5}}
		\]
		In any case we see that there is a number $C>0,$
		\[
		\sum_{k\geq 1}\#_k\geq C(-\log \delta)^{-2}\delta^{-3}.
		\]
		This number $C$ does not depends on $\delta$ and the choices of line segments. This concludes the proof.
	\end{proof}
	
	To conclude this section, we show an almost sharp tubular incidence estimate in $\mathbb{R}^2.$ Let $\{l_i\}_{i\leq L}$ be $L$ many unit segments in $\mathbb{R}^2$ whose directions form a $10\delta$-separated set in $S^1.$ This forces $L\lesssim \delta^{-1}.$ Now consider the case when all the lines are centred at the origin and the lines are as close to each other as possible. Then a similar computation as in the case for $\mathbb{R}^3$ shows that
	\[
	\lambda(P_r)\approx \delta L/r^2,
	\]
	where $\lambda$ is the Lebesgue measure on $\mathbb{R}^2$ and
	\[
	P_r=\{x\in\mathbb{R}^2:x\text{ belongs to at least } r \text{ tubes among } l^{\delta}_i,i\leq L     \}.
	\]
	To see where the factor $\delta L$ comes from, observe that the lines $l_i,i\leq L$ essentially occupy a circular sector with angle $\approx L\delta.$ On the other hand, we have the following result which is basically Cordoba's estimate, see \cite{C74}.
	\begin{thm}
		Let $\{l_i\}_{i\in\{1,\dots,L\}}$ be $L$ unit line segments in $\mathbb{R}^2$ whose the directions of are $\delta$-separated, then there is a constant $C>0$ such that
		\[
		\lambda(P_r)\leq C\delta L\log L/r^2.
		\]
	\end{thm}
	\begin{proof}
		The proof is similar to that of Theorem \ref{weak}. We need to estimate the following sum,
		\[
		\sum_{1\leq i<j\leq L}\lambda(l_i^{3\delta}\cap l_j^{3\delta}).
		\]
		The above sum achieves the largest value when all the lines are as close to each other as possible. Now we take $i=1$ and consider the following sum
		\[
		\sum_{j\geq 2} \lambda(l^{3\delta}_1\cap l^{3\delta}_j).
		\]
		In the worst case, the directions of lines $l_j,j\geq 2$ are of displacement $-(L-1)\delta/2,-(L+1)\delta/2,\dots,(L-1)\delta$ with respect to $l_1.$ For two lines $l_1,l_j$ with direction separation $\Delta$ we have the following estimate
		\[
		\lambda(l^{3\delta}_1\cap l^{3\delta}_j)\leq C' \delta^2/\Delta
		\]
		for an absolute constant $C'>0.$ Therefore we see that
		\[
		\sum_{j\geq 2} \lambda(l^{3\delta}_1\cap l^{3\delta}_j)\lesssim \delta\log L.
		\]
		This shows that
		\[
		\sum_{1\leq i<j\leq L}\lambda(l_i^{3\delta}\cap l_j^{3\delta})\lesssim \delta L\log L.
		\]
		Combining with the argument in the proof of Theorem \ref{weak} we see that
		\[
		\lambda(P_r)\lesssim \delta L\log L/r^2.
		\]
	\end{proof}
	As we have $L\lesssim \delta^{-1},$ we see that the above theorem is almost sharp, up to a $-\log \delta$ factor.
	
	\section{Acknowledgement.}
	HY was financially supported by the university of St Andrews.

	\providecommand{\bysame}{\leavevmode\hbox to3em{\hrulefill}\thinspace}
	\providecommand{\MR}{\relax\ifhmode\unskip\space\fi MR }
	\providecommand{\MRhref}[2]{%
		\href{http://www.ams.org/mathscinet-getitem?mr=#1}{#2}
	}
	\providecommand{\href}[2]{#2}
	
	\bibliographystyle{amsalpha}
	
\end{document}